\documentclass[a4paper,11pt,twoside]{article}

\usepackage{amsmath, amssymb, amsthm, ascmac}
\usepackage{url}
\usepackage{graphicx}

\usepackage{mathptmx}
\newtheorem{thm}{Theorem}[section]
\newtheorem{prop}{Proposition}[section]
\newtheorem{lem}{Lemma}[section]
\newtheorem{cor}{Corollary}[section]

\newtheorem{defn}{Definition}[section]
\newtheorem{asA}{A}
\newtheorem{asB}{B}

\theoremstyle{definition}
\newtheorem{example}{Example}[section]
\newtheorem{remark}{Remark}[section]

\numberwithin{equation}{section}
\allowdisplaybreaks

\def\R{\mathbb{R}}

\def\l{\left}
\def\r{\right}

\def\p{\partial}

\def\toP{\stackrel{p}{\longrightarrow}}
\def\toD{\stackrel{d}{\longrightarrow}}

\def\E{\mathbb{E}}

\def\wh#1{\widehat{#1}}
\def\ol#1{\overline{#1}}

\def\Var{\mathrm{Var}}
\def\Cov{\mathrm{Cov}}
\def\df{\mathrm{d}}

\newcommand{\bi}{\begin{itemize}}
\newcommand{\ei}{\end{itemize}}

\title{Local increment inference for time-inhomogeneous drift in Gaussian processes}
\author{Yasutaka Shimizu\footnote{Department of Applied Mathematics, Waseda University. E-mail: {\tt shimizu@waseda.jp}}}
\date{\today}

\begin{document}

\maketitle 
\begin{abstract}
We study statistical inference for deterministic drifts in Gaussian process models under high-frequency observations over an expanding time horizon.
Using a least squares-type contrast based on first-order increments, we establish consistency and asymptotic normality under conditions on drift accumulation and increment dependence. A key feature is that the convergence rate is determined jointly by the deterministic signal and the full covariance structure of the weighted Gaussian increments, rather than by local noise roughness alone. For power and fixed-frequency periodic drifts under Gaussian and Ornstein--Uhlenbeck covariance kernels, we derive explicit convergence rates and limiting variances, revealing distinct regimes depending on the drift structure and, for periodic drifts, the noise spectrum.
These results clarify the respective roles of sampling frequency and observation horizon.

\begin{flushleft}
{\it Keywords:} Gaussian processes, high-frequency data, time-inhomogeneous drift, convergence rates, dependent increments
\vspace{1mm}\\
{\it MSC2010:} {\bf 62M10}; 62F12,  60G15.
\end{flushleft}
\end{abstract} 

\section{Introduction}

Gaussian processes (GPs) provide a flexible probabilistic framework for modeling time-dependent phenomena with uncertainty quantification.
Because of their analytical tractability and nonparametric structure, they are widely used in statistics, machine learning, spatial statistics, and time series analysis.
See, for example, Ibragimov and Rozanov~\cite{ir78} and Rasmussen and Williams~\cite{rw06}.

In many practical situations, observed data exhibit both deterministic long-term trends and stochastic short-term fluctuations.
This motivates combining a deterministic drift that varies over time with a stationary Gaussian component.

In this paper, we consider a stochastic process of the form
\begin{align}
X_t = Z_t + \int_0^t \mu(s)\,\df s, \quad t\ge0, \label{true}
\end{align}
where $\mu$ is a deterministic mean density, referred to as the drift, and $Z=(Z_t)_{t\ge0}$ is a centered stationary Gaussian process with covariance kernel $K$ satisfying
\[
\E[Z_tZ_s]=K(|t-s|), \quad t,s\ge0.
\]
The process is observed at time points $t_i=ih_n\ (i=0,1,\dots,n)$ under high-frequency sampling:
\[
h_n\to0, \quad T_n:=nh_n\to\infty\quad (n\to \infty).
\]
The primary objective of the present paper is statistical inference for the drift $\mu$ from high-frequency observations of the process $X$.

In many empirical analyses of Gaussian process models, the observed data are first centered by subtracting the sample mean and then treated as approximately mean-zero observations.
Such preprocessing is convenient for covariance analysis and kernel estimation, particularly in stationary settings.

However, subtracting a single empirical mean does not generally remove a time-varying deterministic trend.
The centered observations may therefore retain a systematic time-dependent mean, so treating them as observations of a centered stationary process can obscure the distinction between deterministic structure and stochastic dependence.
This issue is especially relevant when the observation horizon diverges and the trend cannot be represented by a constant level.
We therefore formulate inference directly for the deterministic drift, rather than regarding it as a nuisance component to be removed by empirical centering.

Our approach is based on the local increments
\[
\Delta_i^n X := X_{t_i}-X_{t_{i-1}},\quad i=1,2,\dots,n,
\]
and a least squares-type contrast, a standard tool for trend estimation.
Subtracting a common empirical mean leaves these increments unchanged; the purpose of our approach is to model their time-dependent deterministic component explicitly.

Parameter estimation for Gaussian processes has been extensively studied.
Classical approaches include maximum likelihood estimation, composite likelihood methods, and frequency-domain procedures; see Rasmussen and Williams~\cite{rw06}, Varin et al.~\cite{vetal11}, Whittle~\cite{w51}, Fukasawa and Takabatake~\cite{ft16}, and Takabatake~\cite{t22}, among others.
Much of this literature focuses on covariance structures, spectral characteristics, or likelihood approximation techniques for centered Gaussian processes.

Deterministic mean estimation also has a substantial regression literature.
For regression with dependent errors, Phillips~\cite{p22} studies boundary effects caused by differencing in polynomial regression, and Konev and Pergamenchtchikov~\cite{kp10} study periodic regression with Gaussian noise.
Kobayashi et al.~\cite{ketal25} studied maximum likelihood estimation for mean functions of Gaussian processes under small-noise asymptotics through continuous-time likelihood analysis and also developed M-estimators based on discrete observations.

The present paper uses a least squares-type contrast based on local increments and studies a different asymptotic regime: the observation horizon diverges and the sampling mesh shrinks, while the noise covariance remains fixed.
Specifying the mean function is often essential for predicting future phenomena from real data via Gaussian process models.
Rather than treating the deterministic trend solely as a nuisance component, we regard it as the primary statistical signal of interest.

Our methodology is closely related to local inference methods for diffusion-type processes.
In diffusion statistics, local increment structures play a fundamental role in constructing tractable estimators under high-frequency observations; see, for example, Kessler~\cite{k97} and references therein.
The present paper extends a similar local inference philosophy to Gaussian process models with deterministic nonstationary drifts.
We examine integrable, polynomial-type, and fixed-frequency periodic drift structures, distinguishing their deterministic accumulation properties from the additional covariance conditions needed for consistency and asymptotic normality.

A key feature of the proposed approach is that the asymptotic analysis is formulated directly through covariance structures of increments and Gaussian score distributions, without requiring a mixing-based central limit theorem.
The general results separate deterministic drift accumulation, bounds on increment dependence, and the limiting covariance of the full weighted increment sum.
In particular, off-diagonal covariance terms may cancel the diagonal contribution or change its order, so a convergence rate cannot be inferred from local increment variance alone.
Corollary~\ref{cor:full-score} allows the normalization to be based on a verified full score covariance scale without changing the estimator.

The concrete rate calculations are given in Propositions~\ref{prop:power-rates} and \ref{prop:periodic-rates} for power and fixed-frequency periodic drifts under Gaussian and Ornstein--Uhlenbeck (OU) noise.
For power drifts $\mu_\xi(t)=\xi(1+t)^m$ in the range $m\le1$ considered here, endpoint cancellation produces different rates for growing, constant, and decaying drifts, with a consistency boundary at $m=-1/2$.
For fixed-frequency periodic drifts, the normalization is $\sqrt{T_n}$ and the limiting covariance depends on the noise spectrum at the drift frequency.
Table~\ref{tab:drift-rates} summarizes these rates and sufficient mesh conditions.
Despite their different local smoothness, the Gaussian and OU kernels yield the same normalizations in these examples.
Thus the observation horizon, drift shape, and full increment covariance jointly govern stochastic accuracy, while the sampling mesh also controls discretization error.
The inconsistency result concerns the proposed increment estimator; it is not a general impossibility result, and no optimality claim is made.

The remainder of this paper is organized as follows.
Section~\ref{sec:models} introduces the parametric model, the estimator, and its assumptions.
Section~\ref{sec:main} presents the general consistency and asymptotic normality results, while Section~\ref{sec:example} provides examples of covariance kernels and drift accumulation, together with explicit convergence rate calculations.
Section~\ref{sec:conclusion} concludes the paper with several remarks and possible future extensions.
All proofs are collected in Section~\ref{sec:proofs}.

\subsection*{Notation}

\begin{itemize}

\item If a random vector $Z$ follows the normal distribution with mean vector $m$ and covariance matrix $\Sigma$, we write $Z\sim {\cal N}(m,\Sigma)$.

\item For a centered stationary Gaussian process $Z=(Z_t)_{t\ge 0}$ with covariance kernel $K$, we write $Z\sim GP(0,K)$.

\item For a subset $S\subset \R^p$, $\ol{S}$ is the closure of $S$ with respect to the Euclidean norm.

\item For two nonnegative sequences $a_n$ and $b_n$, we write $a_n\asymp b_n$
if there exist constants $0<c<C<\infty$ such that $cb_n\le a_n\le Cb_n$ for all sufficiently large $n$.
Moreover, we write $a_n\sim b_n$ if $a_n/b_n\to1$.

\item For a function $f(x,y):\R^d\times\R^{d'}\to\R$ and $x=(x_1,\dots,x_d)$, we write
\[
\partial_x f = \left(\frac{\partial f}{\partial x_1},\dots,\frac{\partial f}{\partial x_d}\right)^\top\in\R^d,
\qquad
\partial_x^2 f = \left(\frac{\partial^2 f}{\partial x_i\partial x_j}\right)_{1\le i,j\le d}\in\R^{d\times d},
\]
provided the derivatives exist.

\end{itemize}

\section{Models and Assumptions}\label{sec:models}

For the true process \eqref{true}, we consider a parametric model:
\begin{align}
X_t = Z_t + \int_0^t \mu_\xi(s)\,\df s, \quad X_0 = Z_0, 
\label{model}
\end{align}
where  $\mu_\xi:[0,\infty)\to\R$ is a deterministic drift depending on an unknown parameter $\xi\in \R^p$, 
and $Z=(Z_t)_{t\ge0}\sim GP(0,K)$ is a centered stationary Gaussian process with covariance kernel $K$.

We observe the process at $t_i = ih_n\ (i=0,1,\dots,n)$, under the following {\it high-frequency sampling scheme}:
\begin{align}
h_n \to 0, \quad T_n:=nh_n \to \infty, \quad n\to\infty.
\label{sampling}
\end{align}

We assume that the drift belongs to a parametric family 
\[
\l\{\mu_\xi:[0,\infty)\to\R \,\Big|\, \xi\in\ol{\Xi}, \int_0^t |\mu_\xi(s)|\,\df s < \infty\ \mbox{for any $t>0$}\r\}, 
\]
where $\Xi\subset\R^p$ is an open bounded convex set.
The true parameter is denoted by
\[
\xi_0\in\Xi,\quad \mu_{\xi_0}\equiv \mu.
\]

For $Y=X,Z$, write $\Delta_i^nY=Y_{t_i}-Y_{t_{i-1}}$, $i=1,\ldots,n$. We impose the following assumptions on the Gaussian noise.

\begin{asA}\label{as:increment-covariance}
There exist constants $C>0$ and $\alpha\in(0,2]$ such that
\[
\sup_{1\le i\le n}\sum_{j=1}^n \left|\Cov(\Delta_i^nZ,\Delta_j^nZ)\right|\le C h_n^\alpha.
\]
\end{asA}

\begin{asA}\label{as:local-kernel}
There exist constants $\beta\in(0,2]$ and $c_K>0$ such that
\[
K(0)-K(t)=c_Kt^\beta+o(t^\beta),\quad t\downarrow0.
\]
\end{asA}

Assumption~A\ref{as:increment-covariance} is a short-range dependence condition for local increments.
It controls the cumulative dependence of $\Delta_i^nZ$ on the other increments at the same sampling scale.
For instance, suppose that $K\in C^2([0,\infty))$ with $K'(0+)=0$ and $\int_0^\infty|\p_t^2K(t)|\,\mathrm{d}t<\infty$.
For $k:=|i-j|\ge1$, stationarity yields
\[
\Cov(\Delta_i^nZ,\Delta_j^nZ)=2K(kh_n)-K((k-1)h_n)-K((k+1)h_n).
\]
By the integral representation of the second-order difference,
\[
\left|\Cov(\Delta_i^nZ,\Delta_j^nZ)\right|\le h_n\int_{(k-1)h_n}^{(k+1)h_n}|\p_t^2K(u)|\,\mathrm{d}u.
\]
Since each $u\ge0$ belongs to at most two of the intervals $[(k-1)h_n,(k+1)h_n]$, it follows that
\[
\sum_{k=1}^\infty\left|2K(kh_n)-K((k-1)h_n)-K((k+1)h_n)\right|\le 2h_n\int_0^\infty|\p_t^2K(u)|\,\mathrm{d}u=O(h_n).
\]
The condition $K'(0+)=0$ holds when the even extension of $K$ is differentiable at zero; one-sided differentiability alone does not imply it.
Hence, under these conditions, $|K(0)-K(h_n)|=O(h_n^2)$ and $\Var(\Delta_i^nZ)=O(h_n^2)$.
Consequently,
\begin{equation}
\sup_{1\le i\le n}\sum_{j=1}^n\left|\Cov(\Delta_i^nZ,\Delta_j^nZ)\right|=O(h_n),\label{short-range-c2}
\end{equation}
so that A\ref{as:increment-covariance} holds with $\alpha=1$.

On the other hand, for the Ornstein--Uhlenbeck kernel $K(t)=\sigma^2\exp(-\lambda |t|)$, we have $\beta=1$ and $K(0)-K(t)\sim \sigma^2\lambda t$ as $t\downarrow0$.
Moreover, for $i\ne j$,
\begin{align*}
\Cov(\Delta_i^nZ,\Delta_j^nZ)=\sigma^2\left(2e^{-\lambda |i-j|h_n}-e^{-\lambda |i-j-1|h_n}-e^{-\lambda |i-j+1|h_n}\right). 
\end{align*}
For $|i-j|\ge1$, the absolute value of the right-hand side is bounded by $C h_n^2 e^{-c|i-j|h_n}$.
Therefore,
\begin{align}
\sum_{j\ne i}\left|\Cov(\Delta_i^nZ,\Delta_j^nZ)\right|\le C h_n^2\sum_{k=1}^\infty e^{-ckh_n}\le C h_n.\label{short-range-ou}
\end{align}
Thus A\ref{as:increment-covariance} holds with $\alpha=1$.
Under A\ref{as:local-kernel}, $\Var(Z_{t+h}-Z_t)=2c_Kh^\beta+o(h^\beta)$.

\begin{remark}\leavevmode\ignorespaces\label{rem:alpha-beta}
{A1 controls cumulative increment dependence through $\alpha$, whereas A2 controls single-increment variance through $\beta$. Smooth short-range kernels can have $(\alpha,\beta)=(1,2)$; OU noise has $(1,1)$. {Neither exponent alone determines the full score covariance required in B7.}}
\end{remark}

Next, we shall make assumptions on the parametric drift family.

\begin{asB}\label{as:drift-cont}
For every $t\ge0$, the map $\xi\mapsto\mu_\xi(t)$ is continuous on $\ol{\Xi}$.
\end{asB}

\begin{asB}\label{as:drift-diff}
For every $t\ge0$, the map $\xi\mapsto\mu_\xi(t)$ is twice continuously differentiable on $\ol{\Xi}$.
\end{asB}

\begin{asB}\label{as:time-regularity}
There exists a constant $C>0$ such that
\[
\sup_{\xi\in\ol{\Xi}}\sup_{t\ge0}|\p_t\mu_\xi(t)|\le C.
\]
\end{asB}

\begin{asB}\label{as:separation}
There exists a deterministic sequence $a_n\to\infty$ and a nonnegative continuous function $Q:\overline\Xi\times\overline\Xi\to[0,\infty)$ such that
\[
\sup_{\xi\in\overline{\Xi}}\left|\frac1{a_n}\sum_{i=1}^n\bigl(\mu_\xi(t_{i-1})-\mu_{\xi_0}(t_{i-1})\bigr)^2-Q(\xi,\xi_0)\right|\to0,\quad n\to \infty, 
\]
and the limit function $Q$ satisfies 
\[
Q(\xi,\xi_0)=0 \iff \xi=\xi_0\quad \mbox{on $\overline{\Xi}$}.
\]
\end{asB}

\begin{asB}\label{as:derivative-growth}
There exist a deterministic sequence $b_n\to\infty$ and a measurable function $M:[0,\infty)\to[0,\infty)$ such that
\[
\sup_{\xi\in\ol{\Xi}} |\p_\xi\mu_\xi(t)| \le M(t), \quad t\ge0,
\]
and
\[
\frac1{b_n}\sum_{i=1}^n M(t_{i-1})^2 = O(1),\quad n\to \infty.
\]
\end{asB}

The sequences $a_n$ and $b_n$ measure quadratic drift separation and parameter sensitivity, respectively.

\paragraph{LSE-type contrast functions: }

Under the true parameter, the exact increment decomposition is
\[
\Delta_i^nX=h_n\mu_{\xi_0}(t_{i-1})+r_i^n+\Delta_i^nZ,
\quad r_i^n=\int_{t_{i-1}}^{t_i}\{\mu_{\xi_0}(s)-\mu_{\xi_0}(t_{i-1})\}\,\df s.
\]
Assumption~B\ref{as:time-regularity}, rather than local integrability alone, gives
\[
\sup_{i,\xi}\left|\int_{t_{i-1}}^{t_i}\mu_\xi(s)\,\df s-h_n\mu_\xi(t_{i-1})\right|
\le C\int_0^{h_n}u\,\df u=Ch_n^2/2.
\]
Consequently, the exact marginal distribution at the true parameter is
\[
\Delta_i^nX-h_n\mu_{\xi_0}(t_{i-1})
\sim {\cal N}\bigl(r_i^n,2\{K(0)-K(h_n)\}\bigr).
\]
At a candidate $\xi$, its mean is $r_i^n+h_n\{\mu_{\xi_0}(t_{i-1})-\mu_\xi(t_{i-1})\}$.
Ignoring the $O(h_n^2)$ discretization mean gives an approximate local Gaussian contrast. The marginal variance is independent of $\xi$, so this construction leads to least squares. It does not assert independence of the increments or an exactly centered residual.

This motivates the following contrast function:
\begin{align}
H_n(\xi) := \sum_{i=1}^n\bigl(\Delta_i^n X-\mu_\xi(t_{i-1})h_n\bigr)^2.
\label{contrast}
\end{align}
The estimator is then defined by
\begin{align}
\wh{\xi}_n \in \arg\min_{\xi\in\ol{\Xi}} H_n(\xi).
\label{estimator}
\end{align}

\section{Main Results and Examples}

\subsection{Main Theorems}\label{sec:main}

We first establish consistency of the estimator \eqref{estimator}.

\begin{thm}\label{thm:consistency}
Suppose Assumptions~A\ref{as:increment-covariance}, B\ref{as:drift-cont}, and B\ref{as:time-regularity}--B\ref{as:derivative-growth} hold.
Moreover, assume that
\begin{align}
a_n h_n^{2-\alpha}\to\infty;\quad \frac{nh_n^2}{a_n}\to 0;\quad b_n=O(a_n)
\label{balance-beta}
\end{align}
as $n\to\infty$.
Then $\widehat{\xi}_n$ is weakly consistent:
\[
\widehat{\xi}_n\to^p\xi_0.
\]
\end{thm}

\begin{remark}
The balance $b_n=O(a_n)$ controls the stochastic term uniformly over the parameter set; see Section~\ref{sec:proofs}. The drift examples below explain when this balance holds or fails.
\end{remark}

We next establish asymptotic normality.
To formulate the limiting covariance structure, define
\begin{align}
I_n := \frac1{a_n}\sum_{i=1}^n\p_\xi\mu_{\xi_0}(t_{i-1})\p_\xi\mu_{\xi_0}(t_{i-1})^\top,
\label{information}
\end{align}
where $a_n$ is the sequence given in B\ref{as:separation}. We impose the following additional assumptions. 

\begin{asB}\label{as:information}
There exists a positive definite matrix $I(\xi_0)$ such that
\[
I_n\to I(\xi_0),\quad n\to \infty.
\]
\end{asB}

\begin{asB}\label{as:score-information}
There exists a positive definite matrix $\Gamma(\xi_0)$ such that
\[
\frac1{a_nh_n^\beta}\sum_{i,j=1}^n\partial_\xi\mu_{\xi_0}(t_{i-1})\partial_\xi\mu_{\xi_0}(t_{j-1})^\top\Cov(\Delta_i^nZ,\Delta_j^nZ)\to\Gamma(\xi_0),\quad n\to \infty.
\]
\end{asB}

\begin{asB} \label{as:second-derivative-growth}
There exist a neighborhood $U$ of $\xi_0$ and a measurable function $N:[0,\infty)\to[0,\infty)$ such that
\[
\sup_{\xi\in U}|\p_\xi^2\mu_\xi(t)|\le N(t), \quad t\ge0,
\]
and
\[
\frac1{a_n}\sum_{i=1}^n N(t_{i-1})^2=O(1),\quad n\to \infty.
\]
In addition, take $U$ to be a {bounded convex neighborhood} and assume that there exist $\rho\in(0,1]$ and $L:[0,\infty)\to[0,\infty)$ such that
\[
|\p_\xi^2\mu_\xi(t)-\p_\xi^2\mu_\eta(t)|\le L(t)|\xi-\eta|^\rho,
\quad \xi,\eta\in U,\qquad
\frac1{a_n}\sum_{i=1}^n L(t_{i-1})^2=O(1).
\]
This local uniform regularity is automatic for drift families linear in the parameter, with $N=L=0$.
\end{asB}

B\ref{as:score-information} is an additional full-covariance assumption, not a consequence of A2 and B6. Off-diagonal terms may cancel the diagonal contribution or change its order. The normalization in Theorem~\ref{thm:clt} is conditional on B7; Corollary~\ref{cor:full-score} below allows a different verified score scale without changing the estimator.

The next theorem gives asymptotic normality of the estimator.

\begin{thm}\label{thm:clt}
Suppose Assumptions~A\ref{as:increment-covariance}, A\ref{as:local-kernel}, and B\ref{as:drift-cont}--B\ref{as:second-derivative-growth} hold.
Let $a_n$ and $b_n$ be the sequences appearing in Assumptions~B\ref{as:separation} and B\ref{as:derivative-growth}, respectively, and assume that $b_n=O(a_n)$. Furthermore, assume that 
\begin{align}
a_n h_n^{2-(\alpha\wedge \beta)}\to\infty;\quad  \frac{nh_n^2}{a_n}\to0;\quad nh_n^{4-\beta}\to0 \label{asymp.normal-rate}
\end{align}
as $n\to \infty$.  Then,
\[
\sqrt{a_nh_n^{2-\beta}}(\widehat{\xi}_n-\xi_0)\toD {\cal N}\left(0,I(\xi_0)^{-1}\Gamma(\xi_0)I(\xi_0)^{-1}\right).
\]
\end{thm}

\begin{remark}\leavevmode\ignorespaces
{When $\alpha=\beta$, the exact conditions in \eqref{asymp.normal-rate} remain
\[
a_nh_n^{2-\alpha}\to\infty,\qquad nh_n^2/a_n\to0,\qquad nh_n^{4-\alpha}\to0.
\]
The stronger pair $a_nh_n^{2-\alpha}\to\infty$ and $nh_n^2\to0$ is sufficient, but is {not in general equivalent} to these conditions.
For the smooth kernels below, $\alpha=1$, $\beta=2$, so the conditions are
\[
a_nh_n\to\infty,\qquad nh_n^2/a_n\to0,\qquad nh_n^2\to0.
\]
{In every case, B7 must also be verified.}}
\end{remark}

The same contrast also admits the following extension, which keeps deterministic separation and the full score variance distinct.

\begin{cor}\leavevmode\ignorespaces\label{cor:full-score}
{Suppose A\ref{as:increment-covariance}, B\ref{as:drift-cont}--B\ref{as:information}, and B\ref{as:second-derivative-growth} hold, with $b_n=O(a_n)$ and
\[
a_nh_n^{2-\alpha}\to\infty,\qquad nh_n^2/a_n\to0.
\]
In place of B7, suppose that for some $v_n>0$,
\[
\frac1{v_n}\Var\left(\sum_i\partial_\xi\mu_{\xi_0}(t_{i-1})\Delta_i^nZ\right)\to\Gamma_*>0,
\qquad R_n:=\frac{a_nh_n}{\sqrt{v_n}}\to\infty,
\]
and $na_nh_n^4/v_n\to0$. Then the estimator \eqref{estimator}, with the original left-endpoint contrast, satisfies
\[
R_n(\widehat\xi_n-\xi_0)\toD N(0,I(\xi_0)^{-1}\Gamma_*I(\xi_0)^{-1}).
\]
{The scale $v_n$ need not diverge.} Setting $v_n=a_nh_n^\beta$ recovers the score normalization of Theorem~\ref{thm:clt}.}
\end{cor}

\subsection{Examples}\label{sec:example}

\subsubsection{Examples of covariance kernels}

\begin{example}[Gaussian kernel]\label{ex:gauss-kernel}
Consider the Gaussian kernel
\[
K(t)=a\exp\left(-\frac{b}{2}t^2\right), \quad a,b>0, 
\]
which satisfies A\ref{as:increment-covariance} with $\alpha=1$ due to \eqref{short-range-c2}. 
Moreover, since
\[
K(0)-K(t)=a\left(1-\exp\left(-\frac{b}{2}t^2\right)\right)=\frac{ab}{2}t^2+o(t^2), \quad t\downarrow0,
\]
A\ref{as:local-kernel} holds with
\[
\alpha=1,\quad \beta=2, \quad c_K=\frac{ab}{2}.
\]
\end{example}

\begin{example}[Mat\'ern kernel]\leavevmode\ignorespaces
Consider the Mat\'ern kernel
\[
K(t)=a\frac{2^{1-\nu}}{\Gamma(\nu)}\left(\sqrt{2\nu}\,b|t|\right)^\nu B_\nu\left(\sqrt{2\nu}\,b|t|\right),
\]
where $a,b,\nu>0$ and $B_\nu$ denotes the modified Bessel function of the second kind.

{ Put $\lambda=\sqrt{2\nu}\,b$. The small-argument Bessel expansion gives, for $\nu>1$,
\[
K(0)-K(t)=\frac{a\lambda^2}{4(\nu-1)}t^2+o(t^2),
\quad \alpha=1,\quad\beta=2,\quad c_K=\frac{ab^2\nu}{2(\nu-1)}.
\]
For $0<\nu<1$,
\[
K(0)-K(t)\sim c_Kt^{2\nu},\qquad
c_K=\frac{a\Gamma(1-\nu)}{\nu\Gamma(\nu)}\left(\frac\lambda2\right)^{2\nu},
\quad \beta=2\nu,\quad\alpha=\min(2\nu,1).
\]
Indeed, for $0<\nu<1/2$, the diagonal and lag-one covariances are $O(h_n^{2\nu})$. For lags $k\ge2$, the integral second-difference bound and $|K''(t)|\le Ct^{2\nu-2}$ near zero give a row sum bounded by
\[
C h_n^{2\nu}+C h_n\int_{h_n}^1u^{2\nu-2}\,\df u+O(h_n)=O(h_n^{2\nu}).
\]
For $\nu>1/2$, $K''$ is integrable at zero and at infinity, so the row sum is $O(h_n)$. At $\nu=1/2$ the kernel is exponential, and \eqref{short-range-ou} applies.
At $\nu=1$, $K(0)-K(t)\sim ab^2t^2\log(1/t)$; hence A2 in its pure-power form does not hold, although A1 holds with $\alpha=1$. This boundary case is excluded from Theorem~\ref{thm:clt} as stated.
The expansions follow from the series for the modified Bessel functions; see NIST~\cite{dlmf}.}
Nevertheless, Corollary~\ref{cor:full-score} may still yield asymptotic normality in the boundary case $\nu=1$, provided its full score covariance and remaining conditions are verified.
\end{example}

\begin{example}[Rational quadratic kernel]
Consider the rational quadratic kernel
\[
K(t)=a\left(1+\frac{b^2t^2}{2\gamma}\right)^{-\gamma}, \quad a,b,\gamma>0.
\]
A Taylor expansion gives
\[
K(0)-K(t)=\frac{ab^2}{2}t^2+o(t^2), \quad t\downarrow0.
\]
Hence, with \eqref{short-range-c2}, 
\[
\alpha=1,\quad\beta=2, \quad c_K=\frac{ab^2}{2}.
\]
\end{example}

\begin{example}[Ornstein--Uhlenbeck kernel]\label{ex:ou-kernel}
Consider the exponential kernel
\[
K(t)=ae^{-b|t|}, \quad a,b>0.
\]
Note that $\alpha=1$ from \eqref{short-range-ou}, and that 
\[
K(0)-K(t)=a(1-e^{-b|t|})=ab|t|+o(|t|), \quad t\downarrow0.
\]
Therefore,
\[
\alpha = \beta=1, \quad c_K=ab.
\]
This example illustrates that the present framework naturally includes non-smooth Gaussian kernels.
\end{example}

\subsubsection{Examples of drift accumulation rates}

We next consider several representative drift families satisfying Assumptions~B\ref{as:separation} and B\ref{as:derivative-growth}.

\begin{example}[Directly Riemann integrable drifts]\leavevmode\ignorespaces
We first consider the directly Riemann integrable (DRI) case, which corresponds to localized drift signals.

\begin{defn}\leavevmode\ignorespaces
A nonnegative measurable function $f:[0,\infty)\to[0,\infty)$ is said to be directly Riemann integrable (DRI) if
\[
\lim_{\delta\downarrow0}\delta\sum_{k=1}^\infty\sup_{t\in[(k-1)\delta,k\delta]}f(t)=\lim_{\delta\downarrow0}\delta\sum_{k=1}^\infty\inf_{t\in[(k-1)\delta,k\delta]}f(t)<\infty.
\]
See Feller~\cite{f71} for general properties of directly Riemann integrable functions.
\end{defn}

{ Write $f_\xi(t)=\{\mu_\xi(t)-\mu_{\xi_0}(t)\}^2$.
Assume that each $f_\xi$ is DRI, and that there is a DRI function $G\ge0$ such that
\[
|f_\xi(t)-f_\eta(t)|\le G(t)|\xi-\eta|,\qquad \xi,\eta\in\overline\Xi.
\]
Assume also that $Q(\xi,\xi_0):=\int_0^\infty f_\xi(t)\,\df t$ is zero only at $\xi=\xi_0$.
For fixed $\xi$, the DRI property and $T_n\to\infty$ give
\[
h_n\sum_{i=1}^nf_\xi(t_{i-1})\to Q(\xi,\xi_0).
\]
The discrete sums are uniformly Lipschitz in $\xi$, since $h_n\sum_iG(t_{i-1})=O(1)$; the integral limit is Lipschitz as well. A finite covering of the compact parameter set therefore makes this convergence uniform, proving B4 with $a_n=h_n^{-1}$.
If $M^2$ is DRI, B5 holds with $b_n=h_n^{-1}$.
A sufficient condition for the envelope above is the parameter derivative bound B5 with DRI $M^2$, since convexity yields $G=2\operatorname{diam}(\overline\Xi)M^2$.
Examples include $\mu_\xi(t)=e^{-\xi t}$ and $e^{-\xi t^2}$ when the closure of the parameter interval is contained in $(0,\infty)$.

These are accumulation examples, not automatic applications of the asymptotic normality theorem. For both the Gaussian and OU kernels, $\alpha=1$, and $a_nh_n=1$ fails the sufficient consistency condition. No convergence rate follows from Theorem~\ref{thm:clt} in this case.}
\end{example}

\begin{example}[Polynomial-type drifts]\leavevmode\ignorespaces\label{ex:polynomial}
{ For $q\in\R$, set
\[
A_q(T):=\int_0^T(1+t)^{2q}\,\df t
\asymp\begin{cases}
T^{2q+1},&q>-1/2,\\
\log T,&q=-1/2,\\
1,&q<-1/2.
\end{cases}
\]
If the squared drift separation has polynomial order $(1+t)^{2q}$ and the squared derivative envelope has order $(1+t)^{2r}$, their respective accumulation scales are $h_n^{-1}A_q(T_n)$ and $h_n^{-1}A_r(T_n)$. In particular, $r\le q$ gives $b_n=O(a_n)$. Order comparisons alone do not establish the uniform limit and identifiability required by B4.

A concrete family satisfying B4--B6 is
\[
\mu_\xi(t)=\xi(1+t)^q,\qquad \xi\in\Xi\subset\R,
\]
with fixed $q$. Choose $a_n=b_n=h_n^{-1}A_q(T_n)$. Monotonicity of $(1+t)^{2q}$ bounds the error of the left Riemann sum by $h_n|1-(1+T_n)^{2q}|$, which is $o(A_q(T_n))$. Thus
\[
\frac1{a_n}\sum_{i=1}^n\{\mu_\xi(t_{i-1})-\mu_{\xi_0}(t_{i-1})\}^2
\to(\xi-\xi_0)^2
\]
uniformly on $\overline\Xi$, and $I(\xi_0)=1$. B3 holds for $q\le1$; for $q>1$ the bounded time-derivative assumption fails. B8 holds with zero envelopes because the family is linear in $\xi$.
For $\alpha=1$, the sufficient consistency conditions become $A_q(T_n)\to\infty$ and $T_nh_n^2/A_q(T_n)\to0$. They fail in the localized case $q<-1/2$.
The expression $\{A_q(T_n)h_n^{1-\beta}\}^{1/2}$ is only the candidate normalization under B7; no estimation rate is inferred from accumulation alone. The full covariance calculation is given in Section~\ref{sec:combined}.}
\end{example}

\begin{example}[Periodic-type drifts]\leavevmode\ignorespaces\label{ex:periodic}
{ Consider a fixed period $P>0$ and a drift family jointly continuous on $\overline\Xi\times[0,P]$, periodic in time, with bounded time and parameter derivatives. Suppose
\[
Q(\xi,\xi_0)=\frac1P\int_0^P\{\mu_\xi(t)-\mu_{\xi_0}(t)\}^2\,\df t
\]
is zero only at $\xi_0$. Uniform Riemann approximation and the bounded contribution of an incomplete period prove B4 with $a_n=n$; B5 holds with $b_n=n$.
For example, for known $\omega>0$,
\[
\mu_\xi(t)=\xi_1\sin(\omega t)+\xi_2\cos(\omega t),
\qquad Q(\xi,\xi_0)=\tfrac12|\xi-\xi_0|^2,\quad I(\xi_0)=\tfrac12\mathrm{Id}_2.
\]
For kernels with $\alpha=1$, Theorem~\ref{thm:consistency} applies because $a_nh_n=T_n\to\infty$ and $nh_n^2/a_n=h_n^2\to0$. Asymptotic normality and its rate require the calculation below.

Unknown-frequency families, such as $\mu_\xi(t)=\sin(\xi t)$ on a positive bounded interval, are outside these assumptions. For each fixed $\xi\ne\xi_0$, the normalized continuous squared separation tends to $1$, whereas at $\xi_0$ it is $0$. This discontinuous pointwise limit cannot be the continuous uniform limit required by B4. In addition, parameter derivative accumulation has order $h_n^{-1}T_n^3$, so $b_n=O(a_n)$ fails when $a_n=n$.}
\end{example}

\subsubsection{Drift-specific convergence rates and limiting variances}\label{sec:combined}

Set $w_{i,n}=\p_\xi\mu_{\xi_0}(t_{i-1})$ and
\[
V_n:=\Var\left(\sum_{i=1}^nw_{i,n}\Delta_i^nZ\right)
=\sum_{i,j=1}^nw_{i,n}w_{j,n}^\top\Cov(\Delta_i^nZ,\Delta_j^nZ).
\]
B7 requires $V_n/(a_nh_n^\beta)\to\Gamma(\xi_0)>0$ in the positive definite sense. Neither B5 nor B6 determines $V_n$; the following calculations retain the off-diagonal terms.

\begin{example}[Polynomial drift with Gaussian and Ornstein--Uhlenbeck kernels]\leavevmode\ignorespaces\label{ex:poly-cov}
Consider the polynomial drift
\[
\mu_\xi(t)=\xi(1+t)^m,\quad \xi\in\Xi\subset\R,
\]
{ where $m\le1$ is fixed, so that B3 holds. We retain the left-endpoint contrast \eqref{contrast}. Put $g(t)=(1+t)^m$, $D_n=\sum_i g(t_{i-1})^2$, and
\[
A_m(T)=\int_0^T(1+t)^{2m}\,\df t,\qquad a_n=b_n=h_n^{-1}A_m(T_n).
\]
Then $h_nD_n\sim A_m(T_n)$, and B4 and B6 hold with $Q(\xi,\xi_0)=(\xi-\xi_0)^2$ and $I(\xi_0)=1$. }
{

\begin{prop}[Power-drift rate classification]\leavevmode\ignorespaces\label{prop:power-rates}
For $K(t)=ae^{-bt^2/2}$ or $K(t)=ae^{-b|t|}$, $a,b>0$, let $h_n\to0$, $T_n\to\infty$, and $\xi_0\in\Xi$. For the original estimator \eqref{estimator},
\begin{align*}
T_n^{m+1}(\widehat\xi_n-\xi_0)&\toD N(0,K(0)(2m+1)^2),&&0<m\le1,\\
T_n(\widehat\xi_n-\xi_0)&\toD N(0,2K(0)),&&m=0,\\
T_n^{2m+1}(\widehat\xi_n-\xi_0)&\toD N(0,(2m+1)^2V_{m,\infty}),&&-1/2<m<0,\\
\log T_n(\widehat\xi_n-\xi_0)&\toD N(0,V_{-1/2,\infty}),&&m=-1/2,
\end{align*}
where $V_{m,\infty}=\Var(-Z_0-\int_0^\infty m(1+t)^{m-1}Z_t\,\df t)>0$ for $m<0$. For $m>0$, also assume $h_nT_n^m\to0$. For $m<-1/2$, the unconstrained estimator converges in $L^2$ to $\xi_0+S_{m,\infty}/A_m(\infty)$, where $S_{m,\infty}$ is the centered Gaussian variable defining $V_{m,\infty}$. The constrained estimator converges to its projection onto $\overline\Xi$ and is inconsistent.
\end{prop}

The proof is given in Section~\ref{proof:power-rates}.}
\end{example}

\begin{example}[Periodic drift with Gaussian and Ornstein--Uhlenbeck kernels]\leavevmode\ignorespaces\label{ex:periodic-cov}
\leavevmode\par
{ Let $w(t)=(\sin(\omega t),\cos(\omega t))^\top$, $\omega>0$, and $\mu_\xi(t)=\xi^\top w(t)$. Then $a_n=b_n=n$ and $I(\xi_0)=\mathrm{Id}_2/2$.
For either kernel, define
\[
C_K(\omega):=\int_0^\infty K(u)\cos(\omega u)\,\df u>0.
\]
Here we impose the sufficient sampling condition $T_nh_n\to0$, in addition to $h_n\to0$ and $T_n\to\infty$.
}
{

\begin{prop}[Fixed-frequency periodic rate]\leavevmode\ignorespaces\label{prop:periodic-rates}
For either kernel in Proposition~\ref{prop:power-rates}, let $h_n\to0$, $T_n\to\infty$, $T_nh_n\to0$, and $\xi_0\in\Xi$. For the fixed-frequency drift above, the original estimator satisfies
\[
\sqrt{T_n}(\widehat\xi_n-\xi_0)\toD N(0,4\omega^2C_K(\omega)\mathrm{Id}_2).
\]
Here $C_K(\omega)=ab/(b^2+\omega^2)$ for OU noise and $C_K(\omega)=a\sqrt{\pi/(2b)}e^{-\omega^2/(2b)}$ for Gaussian noise.
\end{prop}

The proof is given in Section~\ref{proof:periodic-rates}.}
\end{example}

Table~\ref{tab:drift-rates} compares the results for the same estimator and the same two noise families. In particular, a constant drift and a periodic drift both have $a_n$ of order $n$, yet their score variances have orders $1$ and $T_n$. Their normalizations therefore differ: $T_n$ versus $\sqrt{T_n}$. This contrast isolates the role of covariance cancellation beyond deterministic accumulation.

For constant drift, $\widetilde\xi_n-\xi_0=(Z_{T_n}-Z_0)/T_n$ for every mesh. Refining the mesh at a fixed horizon leaves this estimator unchanged. For $m<-1/2$, even an expanding horizon together with $h_n\to0$ does not yield consistency of this estimator. These precise statements distinguish the role of high-frequency sampling from that of long-time drift accumulation.

\begin{table}[htbp]
\centering
\begingroup\small
\caption{Normalizations $R_n$ for the original increment estimator under Gaussian or OU noise. In the consistent cases, $R_n(\widehat\xi_n-\xi_0)$ has the nondegenerate normal limit stated in Propositions~\ref{prop:power-rates}--\ref{prop:periodic-rates}. Every row assumes $h_n\to0$ and $T_n\to\infty$.}\label{tab:drift-rates}
\begin{tabular}{lll}
\hline
Drift & $R_n$ & Additional sufficient mesh condition\\
\hline
$(1+t)^m$, $0<m\le1$ & $T_n^{m+1}$ & $h_nT_n^m\to0$\\
Constant ($m=0$) & $T_n$ & None\\
$(1+t)^m$, $-1/2<m<0$ & $T_n^{2m+1}$ & None\\
$(1+t)^{-1/2}$ & $\log T_n$ & None\\
$(1+t)^m$, $m<-1/2$ & Inconsistent & No further mesh restriction\\
Fixed-frequency sine/cosine & $\sqrt{T_n}$ & $T_nh_n\to0$\\
\hline
\end{tabular}
\endgroup
\end{table}

%\section{Numerical experiments}\label{sec:simulation}
\section{Concluding remarks}\label{sec:conclusion}

In this paper, we studied statistical inference for deterministic drift structures in Gaussian process models under high-frequency observations.
The proposed methodology is based on local increment information and a least squares-type contrast constructed from adjacent increments.
The present framework treats the deterministic drift itself as the primary inferential target rather than merely a nuisance component to be removed by empirical centering.

A central feature of the theory is that the asymptotic behavior of the estimator is governed jointly by deterministic drift accumulation and the full covariance structure of the weighted Gaussian increments.
Local noise roughness alone does not determine the convergence rate, because off-diagonal covariance terms may cancel the diagonal contribution or change its order.
The explicit calculations for power and fixed-frequency periodic drifts illustrate the resulting statistical regimes.
For both Gaussian and Ornstein--Uhlenbeck noise, constant and growing power drifts yield faster convergence than fixed-frequency periodic drifts. Consistent decaying power drifts converge more slowly than constant drifts. Compared with the periodic normalization $\sqrt{T_n}$, the normalization $T_n^{2m+1}$ for $-1/2<m<0$ is faster when $-1/4<m<0$, has the same order when $m=-1/4$, and is slower when $-1/2<m<-1/4$. At $m=-1/2$ the normalization is $\log T_n$, and for $m<-1/2$ the estimator is inconsistent.
These results distinguish the role of the observation horizon in stochastic accuracy from that of the sampling mesh in controlling discretization error.
The inconsistency result is specific to the proposed increment estimator and does not establish a general impossibility of estimating localized drifts.

The framework developed here is intentionally formulated under relatively general assumptions.
In particular, directly Riemann integrable drifts appear as one representative class rather than a fundamental assumption of the theory.
The examples distinguish verification of deterministic accumulation conditions from application of the consistency and asymptotic normality results.
In particular, the localized drift examples do not automatically satisfy the sufficient consistency conditions for Gaussian or Ornstein--Uhlenbeck noise.
The full score covariance condition must also be checked before a convergence rate is asserted; the additional corollary provides a normalization adapted to that covariance without changing the estimator.
This formulation allows us to examine a broad class of nonstationary deterministic structures while making the scope of the asymptotic conclusions explicit.

\paragraph{Future work.}
Several possible extensions of the present framework remain for future investigation.

First, it would be interesting to incorporate trajectory fitting-type estimators based on deterministic approximations of stochastic dynamics; see Kasonga~\cite{k90}.
Instead of approximating the local drift contribution by the first-order Euler expansion,
\[
\int_{t_{i-1}}^{t_i} \mu_\xi(s)\,\df s \approx \mu_\xi(t_{i-1})h_n,
\]
one may fit the observed trajectory to the solution of a deterministic differential equation generated by the drift itself.
Such an approach may reduce discretization bias and improve finite-sample performance, especially when the drift function is sufficiently smooth.
However, its stochastic error would depend on the weighting and covariance structure induced by the alternative estimating procedure.
Whether trajectory fitting changes the convergence rate or asymptotic variance therefore requires a separate analysis; the rate calculations for the present increment estimator do not settle this question.
A rigorous comparison between local contrast methods and trajectory fitting procedures remains an interesting direction for future work.

Second, although the present paper focuses primarily on drift estimation, extension to simultaneous inference for both drift and covariance parameters remains an important open problem.
In particular, after obtaining a consistent estimator of the drift parameter, one may construct residual-based estimators for covariance or kernel parameters through residual processes of the form
\[
\widehat{Z}_t = X_t - \int_0^t \mu_{\widehat{\xi}_n}(s)\,\df s.
\]
Such procedures would naturally involve ergodic properties and long-time averaging behavior of the residual process.
The asymptotic interaction between the first-stage drift estimation error and the second-stage kernel estimation remains to be clarified.
Feasible estimation of the full score covariance would also be needed to construct confidence regions from the limiting distributions obtained here.
A systematic two-step asymptotic theory in this direction would provide a more unified framework for simultaneous inference on deterministic and stochastic structures.

Third, although the present paper considers stationary Gaussian processes, the local contrast construction itself does not fundamentally rely on Gaussianity.
It is therefore natural to investigate extensions to more general ergodic stochastic processes, including ergodic diffusion processes and related semimartingale models.
In such settings, local increment structures may still provide tractable estimating equations without requiring global likelihood evaluation.
The present proofs do use Gaussianity, both for uniform control of the stochastic contrast and for the normal limits of deterministic weighted increment sums.
Extensions beyond Gaussian noise would therefore require additional arguments, potentially involving mixing properties, ergodic theorems, or martingale approximations, depending on the model.
We hope that the present work provides a useful starting point for further investigation of drift inference under dependent stochastic environments.

\section{Proofs}\label{sec:proofs}

\subsection{Proof of Theorem~\ref{thm:consistency}}

Define the normalized contrast difference
\begin{align}
\ol{H}_n(\xi):=\frac1{a_nh_n^2}\bigl(H_n(\xi)-H_n(\xi_0)\bigr).
\label{normalized-contrast}
\end{align}
Set $d_i(\xi)=\mu_\xi(t_{i-1})-\mu_{\xi_0}(t_{i-1})$ and let $r_i^n$ be the discretization remainder defined in Section~\ref{sec:models}. Expanding the contrast gives
\begin{align*}
\ol H_n(\xi)
&=\frac1{a_n}\sum_i d_i(\xi)^2-\frac2{a_nh_n}\sum_i d_i(\xi)\Delta_i^nZ-\frac2{a_nh_n}\sum_i d_i(\xi)r_i^n\\
&=:J_{1,n}(\xi)+J_{2,n}(\xi)+J_{3,n}(\xi).
\end{align*}

\noindent \underline{As for $J_{1,n}$:} Assumption~B\ref{as:separation} yields 
\[
\sup_{\xi\in \Xi}\left|J_{1,n}(\xi) -Q(\xi,\xi_0)\right|\to0.
\]

\noindent \underline{As for $J_{2,n}$:}

By the mean value theorem and B\ref{as:derivative-growth},
\begin{align*}
\Var\left(\sum_{i=1}^n\{d_i(\xi)-d_i(\eta)\}\Delta_i^nZ\right)
&\le Ch_n^\alpha |\xi-\eta|^2\sum_{i=1}^nM(t_{i-1})^2 
=O(h_n^\alpha b_n|\xi-\eta|^2).
\end{align*}
Since $b_n=O(a_n)$, the canonical metric of the centered Gaussian field $J_{2,n}$ satisfies
\[
\{\E|J_{2,n}(\xi)-J_{2,n}(\eta)|^2\}^{1/2}
\le C s_n|\xi-\eta|,\qquad s_n=(a_nh_n^{2-\alpha})^{-1/2}.
\]
The canonical covering numbers are at most $C(1+s_n/\varepsilon)^p$, since $\overline\Xi$ is a compact subset of $\R^p$. The field is continuous and $J_{2,n}(\xi_0)=0$, so the Gaussian entropy bound \cite{d67} (see also Vershynin~\cite{vershynin18}, Theorem~8.1.3 and Remark~8.1.5) gives
\[
\E\sup_{\xi\in\overline\Xi}|J_{2,n}(\xi)|
\le C\int_0^{Cs_n}\sqrt{\log\{C(1+s_n/\varepsilon)^p\}}\,\df\varepsilon
\le Cs_n\to0.
\]
Thus $\sup_{\xi\in\overline\Xi}|J_{2,n}(\xi)|\toP0$.

\noindent \underline{As for $J_{3,n}$:} Noting that B\ref{as:time-regularity} implies 
\begin{align}
|r_i^n|\le Ch_n^2,\label{r_i^n} 
\end{align}
we have 
\begin{align*}
\sup_{\xi\in\overline{\Xi}}\left|\frac1{a_nh_n}\sum_{i=1}^nd_i(\xi)r_i^n\right|
&\le\frac{Ch_n}{a_n}\sup_{\xi\in\overline{\Xi}}\sum_{i=1}^n|d_i(\xi)| 
\le\frac{Ch_n n^{1/2}}{a_n}\sup_{\xi\in\overline{\Xi}}\left(\sum_{i=1}^nd_i(\xi)^2\right)^{1/2} \\
&\le C\left(\frac{nh_n^2}{a_n}\right)^{1/2}\to0.
\end{align*}
Consequently,
\[
\sup_{\xi\in\overline{\Xi}}|\ol{H}_n(\xi)-Q(\xi,\xi_0)|\toP0.
\]
Since $Q$ is continuous, $Q(\xi,\xi_0)\ge0$, and
\[
Q(\xi,\xi_0)=0 \iff \xi=\xi_0,
\]
standard M-estimation theory (e.g., van der Vaart~\cite{v98}, Theorem~5.7) yields
\[
\wh{\xi}_n\toP\xi_0.
\]
This completes the proof.

\subsection{Proof of Theorem~\ref{thm:clt}}

Write $w_{i,n}=\partial_\xi\mu_{\xi_0}(t_{i-1})$, $\Psi_n=\partial_\xi H_n$, and
\[
\mathcal J_n=\frac1{2a_nh_n^2}\int_0^1\partial_\xi\Psi_n\{\xi_0+u(\widehat\xi_n-\xi_0)\}\,\df u.
\]
Consistency and Lemma~\ref{lem:hessian-xi} imply that the event $A_n=\{\widehat\xi_n\in\Xi,\ \mathcal J_n\text{ is invertible}\}$ has probability tending to one and $\mathcal J_n\toP I(\xi_0)$. On $A_n$, the score equation and Taylor's formula give, with $r_n=\sqrt{a_nh_n^{2-\beta}}$,
\[
r_n(\widehat\xi_n-\xi_0)
=\mathcal J_n^{-1}\frac1{\sqrt{a_nh_n^\beta}}\sum_iw_{i,n}(\Delta_i^nZ+r_i^n).
\]
By B5, \eqref{r_i^n}, and Cauchy--Schwarz,
\[
\frac{\left|\sum_iw_{i,n}r_i^n\right|}{\sqrt{a_nh_n^\beta}}
\le Ch_n^{2-\beta/2}\sqrt{nb_n/a_n}\to0,
\]
since $b_n=O(a_n)$ and $nh_n^{4-\beta}\to0$. The random vector $(a_nh_n^\beta)^{-1/2}\sum_iw_{i,n}\Delta_i^nZ$ is centered Gaussian and its covariance tends to $\Gamma(\xi_0)$ by B7. Slutsky's theorem proves the stated limit on $A_n$. Its complement is negligible because, for every $\varepsilon>0$,
\[
P\bigl(|r_n(\widehat\xi_n-\xi_0)1_{A_n^c}|>\varepsilon\bigr)\le P(A_n^c)\to0.
\]

\subsection{Proof of Corollary~\ref{cor:full-score}}
{ Consistency follows from Theorem~\ref{thm:consistency}. The Hessian argument in Lemma~\ref{lem:hessian-xi} uses A1, B1--B6, B8, and the two consistency balances; it does not require B7. It therefore gives the same limit $I(\xi_0)$. In the score expansion above, replace $\sqrt{a_nh_n^\beta}$ by $\sqrt{v_n}$. The centered Gaussian term converges by the assumed full covariance limit. The deterministic remainder is bounded by
\[
\frac1{\sqrt{v_n}}\left|\sum_i\partial_\xi\mu_{\xi_0}(t_{i-1})r_i^n\right|
\le C\left(\frac{nb_nh_n^4}{v_n}\right)^{1/2}\to0.
\]
The same Taylor expansion and Slutsky argument, now with $R_n=a_nh_n/\sqrt{v_n}$, prove the assertion.}

\subsection{Proof of Proposition~\ref{prop:power-rates}}\label{proof:power-rates}

{
For the unconstrained minimizer of the original contrast,
\begin{equation}
\widetilde\xi_n-\xi_0=\frac{S_n+e_n}{h_nD_n},\quad
S_n=\sum_i g(t_{i-1})\Delta_i^nZ,\quad
e_n=\sum_i g(t_{i-1})r_i^n.\label{power-score-exact}
\end{equation}
Here $r_i^n$ is the discretization remainder, not zero in general.
Discrete summation by parts gives
\[
S_n=g(t_{n-1})Z_{T_n}-Z_0
-\sum_{i=1}^{n-1}\{g(t_i)-g(t_{i-1})\}Z_{t_i}.
\]
Both kernels are integrable and satisfy $\sup_i\sum_j|K((i-j)h_n)|=O(h_n^{-1})$. For $m>0$, the variance of the last sum is at most
\[
C h_n^{-1}\sum_{i=1}^{n-1}\{g(t_i)-g(t_{i-1})\}^2
\le C\int_0^{T_n}|g'(t)|^2\,\df t=o(T_n^{2m}).
\]
It follows that $\Var(S_n)\sim K(0)T_n^{2m}$. For $m=0$, $S_n=Z_{T_n}-Z_0$ exactly and $\Var(S_n)\to2K(0)$.
For $m<0$, mean-square continuity on compact intervals and $\int_0^\infty|g'(t)|\,\df t<\infty$ give
\[
S_n\to S_{m,\infty}:=-Z_0-\int_0^\infty g'(t)Z_t\,\df t\quad\hbox{in }L^2,
\qquad V_{m,\infty}:=\Var(S_{m,\infty})>0.
\]
Indeed, the tail has $L^2$ norm bounded by $\sqrt{K(0)}$ times the tail variation of $g$, uniformly in the mesh. Positivity follows because both kernels have strictly positive spectral densities and the finite signed measure $-\delta_0-g'(t)\,\df t$ is nonzero.

The deterministic remainder in \eqref{power-score-exact} satisfies
\[
|e_n|\le C h_n\int_0^{T_n}(1+t)^{2m-1}\,\df t
=\begin{cases}O(h_nT_n^{2m}),&m>0,\\O(h_n),&m<0,\end{cases}
\]
and $e_n=0$ for $m=0$. Thus $e_n=o(T_n^m)$ for $0<m\le1$ if $h_nT_n^m\to0$, and $e_n=o(1)$ for $m<0$ under $h_n\to0$ alone. Gaussianity yields the four normal limits in Proposition~\ref{prop:power-rates}. The unconstrained estimator is consistent in these regimes, so it agrees with the constrained estimator with probability tending to one. Under the common sufficient condition $T_nh_n\to0$, all these limits also follow from Corollary~\ref{cor:full-score}, with $v_n=T_n^{2m}$ for $m>0$ and $v_n=1$ for $m\le0$.

For $m<-1/2$, $h_nD_n\to A_m(\infty)$, so \eqref{power-score-exact} gives the stated nondegenerate limit. Projection onto $\overline\Xi$ preserves inconsistency since $\xi_0$ is interior.

These calculations also check the original B7 directly. For OU noise, its denominator is $A_m(T_n)$: the ratio tends to zero for $m\ge-1/2$. For $m<-1/2$ it has a positive limit, but the consistency condition of Theorem~\ref{thm:clt} fails. For Gaussian noise its denominator is $h_nA_m(T_n)$, and under the theorem's condition $T_nh_n\to0$ the ratio diverges. Covariance cancellation under differencing is related to the regression phenomena studied by Phillips~\cite{p22}; here the calculation uses a fixed continuous-time covariance kernel and a shrinking observation mesh.
\hfill$\square$}

\subsection{Proof of Proposition~\ref{prop:periodic-rates}}\label{proof:periodic-rates}

{
We verify the full covariance limit
\begin{equation}
\frac1{T_n}\Var\left(\sum_{i=1}^nw(t_{i-1})\Delta_i^nZ\right)
\longrightarrow\omega^2C_K(\omega)\mathrm{Id}_2.\label{periodic-score-limit}
\end{equation}
To see this, define the deterministic integration-by-parts integral
\[
Y_T:=w(T)Z_T-w(0)Z_0-\int_0^Tw'(t)Z_t\,\df t.
\]
No semimartingale assumption is needed. Summation by parts yields exactly
\[
S_n-Y_{T_n}=\sum_{i=1}^n\int_{t_{i-1}}^{t_i}w'(t)\{Z_t-Z_{t_i}\}\,\df t,
\quad S_n=\sum_{i=1}^nw(t_{i-1})\Delta_i^nZ.
\]
A2 and boundedness of $w'$ imply $\|S_n-Y_{T_n}\|_{L^2}\le CT_nh_n^{\beta/2}$.
Thus $\|S_n-Y_{T_n}\|_{L^2}/\sqrt{T_n}\to0$, since $T_nh_n^\beta\to0$ for $\beta=1,2$ under our sampling condition.
Trigonometric product identities and integrability of $K$ give
\[
\frac1T\Var\left(\int_0^T\cos(\omega t)Z_t\,\df t\right)\to C_K(\omega),
\]
with the same limit for sine and zero limiting cross covariance. Indeed, the nonoscillatory part is $\int_0^T(T-u)K(u)\cos(\omega u)\,\df u$, and the remaining part is bounded by a constant times $\int_0^\infty|K(u)|\,\df u$. Boundary variances and their covariances with the integrals are $O(1)$. This proves \eqref{periodic-score-limit}.

Let $D_n=\sum_iw(t_{i-1})w(t_{i-1})^\top$. For the unconstrained least-squares estimator,
\[
\widetilde\xi_n-\xi_0=(h_nD_n)^{-1}
\left\{S_n+\sum_iw(t_{i-1})r_i^n\right\}.
\]
Since $D_n/n\to\mathrm{Id}_2/2$ and $\sum_iw(t_{i-1})r_i^n=O(T_nh_n)$, the scaled bias is $O(\sqrt{T_n}h_n)=o(1)$. Gaussianity and \eqref{periodic-score-limit} therefore yield
\begin{equation}
\sqrt{T_n}(\widehat\xi_n-\xi_0)
\toD{\cal N}\left(0,4\omega^2C_K(\omega)\mathrm{Id}_2\right).\label{periodic-clt}
\end{equation}
The unconstrained estimator is consistent, so it is in $\Xi$ with probability tending to one, justifying the same limit for $\widehat\xi_n$.

For the OU kernel $K(t)=ae^{-b|t|}$,
\[
C_K(\omega)=\frac{ab}{b^2+\omega^2},\quad
\Gamma(\xi_0)=\frac{ab\omega^2}{b^2+\omega^2}\mathrm{Id}_2,
\quad I^{-1}\Gamma I^{-1}=\frac{4ab\omega^2}{b^2+\omega^2}\mathrm{Id}_2.
\]
Here B7 holds since $a_nh_n^\beta=T_n$, and Theorem~\ref{thm:clt} also applies under the stated sampling condition.
For the Gaussian kernel $K(t)=ae^{-bt^2/2}$,
\[
C_K(\omega)=a\sqrt{\frac\pi{2b}}e^{-\omega^2/(2b)}.
\]
In this case $V_n/(a_nh_n^\beta)\sim h_n^{-1}\omega^2C_K(\omega)\mathrm{Id}_2$, so B7 fails. Equation~\eqref{periodic-clt} also follows from Corollary~\ref{cor:full-score} with $v_n=T_n$ under the stated sampling condition. Both limiting variances depend on the full kernel at frequency $\omega$, not just on $c_K$.
\hfill$\square$}

\subsection{Auxiliary lemmas}

\begin{lem}\label{lem:hessian-xi}
Suppose that the assumptions of Theorem~\ref{thm:clt} or Corollary~\ref{cor:full-score} hold.
Then
\[
\frac1{2a_nh_n^2}\int_0^1 \p_\xi\Psi_n\{\xi_0+u(\widehat{\xi}_n-\xi_0)\}\,\df u\toP I(\xi_0).
\]
\end{lem}

\begin{proof}
Put $\xi_n(u):=\xi_0+u(\widehat{\xi}_n-\xi_0)$ for $0\le u\le1$.
Let $U$ be the neighborhood of $\xi_0$ in Assumption~B\ref{as:second-derivative-growth}.
Choose $\eta>0$ such that $\{\xi\in\R^p:|\xi-\xi_0|<\eta\}\subset U$.
Define
\[
C_n:=\{\xi_n(u)\in U \text{ for all } u\in[0,1]\}.
\]
Since $\sup_{0\le u\le1}|\xi_n(u)-\xi_0|\le|\widehat{\xi}_n-\xi_0|$ and $\widehat{\xi}_n\toP\xi_0$, we have $P(C_n)\to1$.
Recall that
\[
H_n(\xi)=\sum_{i=1}^n\{\Delta_i^nX-h_n\mu_\xi(t_{i-1})\}^2.
\]
Then
\[
\Psi_n(\xi)=-2h_n\sum_{i=1}^n\{\Delta_i^nX-h_n\mu_\xi(t_{i-1})\}\p_\xi\mu_\xi(t_{i-1}).
\]
Hence
\[
\p_\xi\Psi_n(\xi)=2h_n^2\sum_{i=1}^n\p_\xi\mu_\xi(t_{i-1})\p_\xi\mu_\xi(t_{i-1})^\top-2h_n\sum_{i=1}^n\{\Delta_i^nX-h_n\mu_\xi(t_{i-1})\}\p_\xi^2\mu_\xi(t_{i-1}).
\]
Therefore,
\[
\frac1{2a_nh_n^2}\int_0^1\p_\xi\Psi_n(\xi_n(u))\,\df u=B_n+R_n,
\]
where
\[
B_n:=\frac1{a_n}\sum_{i=1}^n\int_0^1\p_\xi\mu_{\xi_n(u)}(t_{i-1})\p_\xi\mu_{\xi_n(u)}(t_{i-1})^\top\,\df u
\]
and
\[
R_n:=-\frac1{a_nh_n}\sum_{i=1}^n\int_0^1\{\Delta_i^nX-h_n\mu_{\xi_n(u)}(t_{i-1})\}\p_\xi^2\mu_{\xi_n(u)}(t_{i-1})\,\df u.
\]
First, we show that $B_n\toP I(\xi_0)$.
On $C_n$, by the mean value theorem, Assumption~B\ref{as:derivative-growth}, and Assumption~B\ref{as:second-derivative-growth}, we have
\[
\left|B_n-\frac1{a_n}\sum_{i=1}^n\p_\xi\mu_{\xi_0}(t_{i-1})\p_\xi\mu_{\xi_0}(t_{i-1})^\top\right|\le C|\widehat{\xi}_n-\xi_0|\frac1{a_n}\sum_{i=1}^nM(t_{i-1})N(t_{i-1}).
\]
By the Cauchy--Schwarz inequality,
\[
\frac1{a_n}\sum_{i=1}^nM(t_{i-1})N(t_{i-1})\le\left\{\frac1{a_n}\sum_{i=1}^nM(t_{i-1})^2\right\}^{1/2}\left\{\frac1{a_n}\sum_{i=1}^nN(t_{i-1})^2\right\}^{1/2}=O(1).
\]
Since $\widehat{\xi}_n\toP\xi_0$, it follows that
\[
\left|B_n-\frac1{a_n}\sum_{i=1}^n\p_\xi\mu_{\xi_0}(t_{i-1})\p_\xi\mu_{\xi_0}(t_{i-1})^\top\right|1_{C_n}\toP0.
\]
By the definition of $I_n$ and Assumption~B\ref{as:information},
\[
\frac1{a_n}\sum_{i=1}^n\p_\xi\mu_{\xi_0}(t_{i-1})\p_\xi\mu_{\xi_0}(t_{i-1})^\top=I_n\to I(\xi_0).
\]
Thus $B_n1_{C_n}\toP I(\xi_0)$.
Next, we show that $R_n=o_P(1)$ on $C_n$.
Write
\[
\Delta_i^nX-h_n\mu_{\xi_n(u)}(t_{i-1})=\Delta_i^nZ+\left\{\int_{t_{i-1}}^{t_i}\mu_{\xi_0}(s)\,\df s-h_n\mu_{\xi_0}(t_{i-1})\right\}+h_n\{\mu_{\xi_0}(t_{i-1})-\mu_{\xi_n(u)}(t_{i-1})\}.
\]
Correspondingly, write $R_n=R_{n,1}+R_{n,2}+R_{n,3}$.
For the Gaussian part, introduce the deterministic-index Gaussian field
\[
G_n(\xi)=\frac1{a_nh_n}\sum_{i=1}^n\p_\xi^2\mu_\xi(t_{i-1})\Delta_i^nZ,\quad\xi\in U.
\]
For each matrix entry, A1 and B8 imply
\[
\sup_{\xi\in U}\E|G_n(\xi)|^2\le\frac{C}{a_nh_n^{2-\alpha}},\qquad
\E|G_n(\xi)-G_n(\eta)|^2\le\frac{C|\xi-\eta|^{2\rho}}{a_nh_n^{2-\alpha}}.
\]
The canonical metric therefore has covering numbers bounded by $C(1+s_n/\varepsilon)^{p/\rho}$, where $s_n=(a_nh_n^{2-\alpha})^{-1/2}$. For each entry $G_{n,rs}$, the variance bound gives $\E|G_{n,rs}(\xi_0)|\le Cs_n$. Applying Dudley's entropy inequality (see Dudley~\cite{d67}; see also Vershynin~\cite{vershynin18}, Theorem~8.1.3 and Remark~8.1.5) to the field anchored at $\xi_0$ gives
\begin{align*}
\E\sup_{\xi\in U}|G_{n,rs}(\xi)|
&\le \E|G_{n,rs}(\xi_0)|
 +\E\sup_{\xi\in U}|G_{n,rs}(\xi)-G_{n,rs}(\xi_0)|\\*
&\le Cs_n+C\int_0^{Cs_n}\sqrt{\log\{C(1+s_n/\varepsilon)^{p/\rho}\}}\,\df\varepsilon
\le Cs_n.
\end{align*}
Summing over the finitely many entries yields $\E\sup_{\xi\in U}|G_n(\xi)|\le Cs_n$ for a matrix norm.
Continuity ensures separability and measurability of these suprema. On $C_n$, $|R_{n,1}|\le\sup_{\xi\in U}|G_n(\xi)|$. Since either set of assumptions implies $a_nh_n^{2-\alpha}\to\infty$, we obtain $R_{n,1}1_{C_n}\toP0$. This uses $\alpha$, not $\beta$, and does not move the random coefficients outside an expectation.

For the time-discretization part, Assumption~B\ref{as:time-regularity} gives
\[
\left|\int_{t_{i-1}}^{t_i}\mu_{\xi_0}(s)\,\df s-h_n\mu_{\xi_0}(t_{i-1})\right|\le Ch_n^2.
\]
Hence, on $C_n$,
\[
|R_{n,2}|\le Ch_n\frac1{a_n}\sum_{i=1}^nN(t_{i-1})\le C\left(\frac{nh_n^2}{a_n}\right)^{1/2}\left\{\frac1{a_n}\sum_{i=1}^nN(t_{i-1})^2\right\}^{1/2}.
\]
Thus $R_{n,2}1_{C_n}\toP0$.
For the parameter-shift part, the mean value theorem gives
\[
\sup_{0\le u\le1}|\mu_{\xi_n(u)}(t_{i-1})-\mu_{\xi_0}(t_{i-1})|\le M(t_{i-1})|\widehat{\xi}_n-\xi_0|.
\]
Therefore, on $C_n$,
\[
|R_{n,3}|\le C|\widehat{\xi}_n-\xi_0|\frac1{a_n}\sum_{i=1}^nM(t_{i-1})N(t_{i-1}).
\]
By the Cauchy--Schwarz inequality and the growth bounds for $M$ and $N$, the last factor is $O(1)$.
Since $\widehat{\xi}_n\toP\xi_0$, we obtain $R_{n,3}1_{C_n}\toP0$.
Consequently, $R_n1_{C_n}=o_P(1)$.
Combining the convergence of $B_n$ and $R_n$, we obtain
\[
\left\{\frac1{2a_nh_n^2}\int_0^1\p_\xi\Psi_n\{\xi_0+u(\widehat{\xi}_n-\xi_0)\}\,\df u-I(\xi_0)\right\}1_{C_n}\toP0.
\]
Since $P(C_n)\to1$, we conclude that
\[
\frac1{2a_nh_n^2}\int_0^1\p_\xi\Psi_n\{\xi_0+u(\widehat{\xi}_n-\xi_0)\}\,\df u\toP I(\xi_0).
\]
\end{proof}

\noindent\begin{minipage}{\textwidth}
\section*{Statements and Declarations}

\paragraph{Funding}
This work is partially supported by the Japan Society for the Promotion of Science (JSPS) KAKENHI, Grant-in-Aid for Scientific Research (B) (Grant No.~24K02907), Grant-in-Aid for Scientific Research (C) (Grant No.~24K06875), and the Japan Science and Technology Agency (JST) CREST (Grant No.~JPMJCR2115).

\paragraph{Competing interests}
The author has no relevant financial or non-financial interests to disclose.

\end{minipage}

\end{document}